\documentclass[11pt]{article}
\usepackage[top=1in, bottom=1in, left=1in, right=1in]{geometry}

\usepackage{./dsm5_header}
\usepackage{amssymb}
\usepackage{amsmath}
\usepackage{accents}
\usepackage{graphicx}

\def\N{\mathbb{N}}
\def\Z{\mathbb{Z}}
\def\C{\mathbb{C}}
\newcommand{\real}{\mathbb{R}}
\newcommand{\complex}{\mathbb{C}}
\def\R{\real}
\def\i{\mathrm{i}}
\def\p{\mathrm{p}}

\newcommand{\cw}{\mathcal{W}}

\newcommand{\cu}{\mathcal{U}}

\newcommand{\cKN}{\hat{\mathcal{K}}_{\p,N}(s)}

\newlength{\dhatheight}
\newcommand{\dhat}[1]{%
    \settoheight{\dhatheight}{\ensuremath{\hat{#1}}}%
    \addtolength{\dhatheight}{-0.35ex}%
    \hat{\vphantom{\rule{1pt}{\dhatheight}}%
    \smash{\hat{#1}}}}

\renewcommand{\Re}{\mathrm{Re}\,}
\renewcommand{\Im}{\mathrm{Im}\,}

\newcommand{\change}[1]{{{}#1}}
\newtheorem{theorem}{Theorem}
\newtheorem{corollary}[theorem]{Corollary}
\newtheorem{lemma}[theorem]{Lemma}
\newtheorem{remark}[theorem]{Remark}

\begin{document}
\setTitle{\textcolor{cyan}{The Time Domain Lippmann-Schwinger Equation and Convolution Quadrature}}
\addAuthor{Armin Lechleiter\footnote[1]{Center for Industrial Mathematics, University of Bremen, 28359 Bremen, Germany.  {\tt e-mail:} lechleiter$@$math.uni-bremen.de. } }
\addAuthor{Peter Monk\footnote[2]{Department of Mathematical Sciences, University of Delaware, Newark DE 19716, USA.  {\tt e-mail:} monk$@$math.udel.edu. }}
\placeTitlePage

\begin{outline}
We consider time domain acoustic scattering from a penetrable medium with a variable sound speed.  This  problem can be reduced to solving a time domain volume Lippmann-Schwinger integral equation.  Using convolution quadrature in time and trigonometric collocation in space we can compute an approximate  solution.  We prove that the time domain Lippmann-Schwinger equation has a unique solution and prove conditional convergence and error estimates for the fully discrete solution for smooth sound speeds.  Preliminary numerical results show that the method behaves well even for discontinuous sound speeds.
\end{outline}

\tableofcontents

\section{Introduction}
The problem we shall study  is to compute the acoustic field scattered by a bounded and Lipschitz smooth scatterer using time domain integral equations. For simplicity we will present the theory in $\real^3$ although the theory can also be verified in $\real^2$, and our numerical results are computed in $\real^2$. We denote position by $x\in\real^3$ and time by $t$. The background medium outside the  \change{scatterer} is taken  to have a constant wave speed $c(x)=c_0>0$.  The scatterer is assumed to be penetrable, and within the scatterer the wave speed $0<c(x)<c_0$ can vary with position. We define $D$ to be the interior of the  support of $c-c_0$, assumed to be Lipschitz smooth.

Denoting the total pressure field by $u=u(x,t)$ we want to solve the wave equation
\[
\frac{1}{c(x)^2}u_{tt}=\Delta u\mbox{ in }\real^3\times\real_+,
\]where $\real_+=\{t\in\real\;|\;t>0\}$.
The field $u$ is assumed to consist of an incident field $u^i$ and a scattered field $u^s$ ($u=u^s+u^i$).  We assume that the incident field $u^i$ is a smooth solution of the  wave equation in the background medium:
\[
\frac{1}{c^2_0}u^i_{tt}=\Delta u^i \mbox{ in }\real^3\times\real.
\]
In addition we assume that $u^i$ is causal so that $u^i=0$ on $D$ for $t\leq 0$.  Then the scattered field $u^s$ vanishes before $t=0$ and we impose the initial conditions
\begin{equation}
u^s=u^s_t=0 \mbox{ in }\real^3\mbox{ at } t=0.
\label{ic}
\end{equation}
These assumptions rule out incident fields due to point sources, but, at the expense of some slightly more complicated notation, it is easy to extend the theory to allow point sources located outside $D$.

It is convenient to define the contrast
\[
q_c(x)=\frac{c_0^2}{c^2(x)}-1.
\]
We shall assume that $q_c\in L^{\infty}(D)$ and there exists a constant $q_{c,+}$ such that
\[
0\leq q_c(x)\leq q_{c,+}<\infty\mbox{ a.e. in } D.
\]
In addition we assume $q_c$ is a weight function on $L^2(D)$ so that if $v\in L^2(D)$ and $\Vert q_c^{1/2}v\Vert_{L^2(D)}=0$ then $v=0$. Later we will make further regularity assumptions on $q_c$.

Then using the fact that $u^i$ satisfies the wave equation we see that $u^s$ satisfies
\begin{eqnarray}
\frac{1}{c^2}u^s_{tt}-\Delta u^s=-\frac{1}{c_0^2}q_cu_{tt}^i\mbox{ in }\real^d \times\real_+,\label{useq}
\end{eqnarray}
together with the initial conditions from (\ref{ic}).  Note that the source term on the right-hand side vanishes outside $D$ since $q_c=0$ there.

We now give a formal description of how to recast the wave equation as a space-time Lippmann-Schwinger integral equation. Later we will prove that this problem has a unique solution in a suitable function space. Denote by $k(x,t)$ the fundamental solution of the wave equation in the background medium given by
\[
k(x,t)=\frac{\delta(t-\Vert x\Vert/c_0)}{4\pi\Vert x\Vert}\mbox{ in }\real^3 \times\real.
\]
For a function $f\in C_0^{\infty}(D\times \real)$ we define the retarded volume potential $V$ by
\begin{equation}
(V(f))(x,t)=\int_\real\int_Dk(x-y,t-\tau)f(\tau,y)\,dV \,d\tau\mbox{ for } x\in\real^3,\; t\in\real.\label{VTdef}
\end{equation}
It is well known that if $w=V(f)$ then $w$ is a  solution of the wave equation
\[
\frac{1}{c_0^2}w_{tt}-\Delta w=f\mbox{ in } \real^3\times\real.
\]
By rewriting equation (\ref{useq}) we see that the scattered field $u^s$ satisfies
\[
\frac{1}{c_0^2}u_{tt}^s-\Delta u^s=
-\frac{1}{c_0^2}q_c(u_{tt}^s+u_{tt}^i)
\]
so that
\[
u^s+\frac{1}{c_0^2}V(q_c(u_{tt}^s+u_{tt}^i))=0,
\]
which gives rise to the Lippmann-Schwinger equation in the time domain:
Find $u^s$ in a suitable function space to be described shortly such that
\begin{equation}
u^s+\frac{1}{c_0^2}V(q_cu^s_{tt})=-\frac{1}{c_0^2}V(q_cu^i_{tt}) \mbox{ in } \real^3\times\real.
\label{ust}
\end{equation}
Note that we shall show that because $u^i=0$ in $D$ for $t\leq 0$, then $u^s=0$ in $D$ for $t\leq 0$ so 
a solution of (\ref{ust}) satisfies (\ref{ic}) automatically.

We shall analyze this integral equation  via the Fourier-Laplace transform \cite{BHa86,lubich}.  Let the transform parameter $s=\sigma+i\omega \in \complex$, where $\sigma\in \real$, $\sigma>\sigma_0>0$ for some constant $\sigma_0$ and $\omega\in \real$. Defining formally
\[
\hat{u}:=\hat{u}(x,s)=\int_{0}^\infty u(x,t)\exp(-st)\,dt
\]
(and similarly $\hat{u}^i$, etc.), we see that 
for any suitably smooth causal function $f$ on $D\times \real_+$, the convolution structure of the time integral in
(\ref{VTdef}) implies that
\begin{equation}
\widehat{V(f)}(x)=(\hat{V}(\hat{f}))(x):=\int_D\Phi(x,y)\hat{f}(y)\,dV,\label{Vhatdef}
\end{equation}
where $\Phi$ is the fundamental solution of the Helmholtz equation
\[
\Phi(x,y)=\frac{\exp(-s\Vert x-y\Vert/c_0)}{4\pi\Vert x-y\Vert}, \ x \not = y \in \real^3.
\]
If $\hat{f}\in L^2(D)$ is extended by zero to $\real^3$ it is well known \cite{colton+kress3} that $\hat{w}=\hat{V}(\hat{f})$ satisfies
\begin{equation}
-\Delta \hat{w}+\frac{s^2}{c_0^2}\hat{w}=\hat{f} \mbox{ in } \real^3 \label{Vhatprop}
\end{equation}
and belongs to $H^1(\real^3)$ since it decays exponentially for large $\Vert x\Vert$.
Taking the Fourier-Laplace transform of (\ref{ust}) we obtain the Fourier-Laplace domain Lippmann-Schwinger equation: find $\hat{u}^s$ such that
\begin{equation}\label{ushateq}
\hat{u}^s+\frac{s^2}{c_0^2}\hat{V}(q_c \hat{u}^s)=-\frac{s^2}{c_0^2}\hat{V}(q_c\hat{u}^i) \mbox{ in } \real^3.
\end{equation}

Convolution quadrature (CQ) provides a way to discretize (\ref{ust}) in time and criteria for choosing appropriate
underlying time stepping schemes \cite{lubich}.  The goal of this paper is to prove convergence of a fully discrete collocation scheme involving trigonometric polynomials to discretize in space and convolution quadrature to discretize in time.  Applying Lubich's theory \cite{lubich}, this involves analyzing the Laplace transformed problem (\ref{ushateq}) and proving that the solution operator is bounded uniformly by some power of the transform parameter $s$ for $s$ in an appropriate part of the complex plane. Because we are using a collocation scheme in space for efficiency, we obtain only conditional convergence of the fully discrete scheme. If we had used a Galerkin scheme we would have unconditional convergence but efficient implementations would be more involved.

The application of convolution quadrature to time dependent boundary integral equations was first suggested and analyzed by Lubich \cite{lubich} and applied in fluid dynamics by Schanz  and coworkers (see \cite{schanz05} and references therein).  Since 2008 and the publication of \cite{WKress1,Wkress2} there has been a major increase in efforts to implement and understand convolution quadrature applied to a variety of problems, see for example  
\cite{schanz12,XWM07,Banjai_2010}. In addition the use of implicit Runge-Kutta integration techniques is now well
understood \cite{ban11} (for simplicity, we shall not use Runge-Kutta methods in this paper). There has been significant progress on extending the analysis of the method to more general
boundary conditions \cite{lal09} and some progress in analyzing electromagnetic problems \cite{che11,bal13}. Of special
importance to us is the work of Banjai and Sauter who show how to compute convolution quadrature solutions
via the solution of several Laplace domain problems and an inverse transform \cite{Banjai_2008}.  All these efforts involve
the use of time domain boundary integral equations.

As described above, in order to treat a spatially varying sound speed we will apply convolution quadrature to a volume integral equation.  To our knowledge this is the first such application.  Discretization of the volume integral equation involves either inverting a volume integral operator at each time step, or alternatively solving a Laplace domain integral equation at several frequencies.  Using the first ``marching on in time'' approach, a straightforward Galerkin scheme based on piecewise polynomials gives unconditional and optimal convergence. But applying the volume integral operator many times might be time consuming, and storing the past history of the solution might become prohibitive.   In this paper we use the  multi-frequency approach, and so the Fourier-Laplace domain operator must be inverted
at many frequencies.  However because this integral operator is of the second kind, efficient solution strategies are possible.  We use trigonometric polynomials to discretize in space (after first periodizing the problem), and then collocate the resulting equations using the techniques from \cite{vain00,hohage01}.  The use of trigonometric polynomials diagonalizes the integral operator, and the integral equation system can be solved efficiently by a two-grid scheme \cite{hohage01}. 

At first sight the use of volume integral equations may appear unattractive compared to a coupled finite-element and boundary integral equation (FE-BIE) approach such as used in  \cite{ban14}. In the FE-BIE method a CQ boundary integral formulation is coupled to an explicit finite element solver in the volume, thus avoiding the storage of past solutions in the volume.   This approach may be preferable if there are discontinuities in the contrast. However, if the contrast is globally smooth, the volume integral equation approach may be useful.  Another case in which volume integral equations might be attractive is for thin structures.

The paper proceeds as follows.  In the next section we give a brief formal derivation that shows how the convolution
quadrature method arises, and explain the relevance of the Fourier-Laplace transform in the analysis of the method.  Then in Section \ref{note} we summarize some notation and spaces related to the time domain problem, and prove two basic results
concerning the mapping properties of the volume potential operator.  In Section \ref{laplace} we analyze the Fourier-Laplace domain
integral equation problem and prove relevant estimates that allow us to use Lubich's theory to estimate the time discretization error \cite{lubich}.
We then show, in Section \ref{trig},  how to periodize the integral equation to enable the use of a trigonometric collocation method to solve the Fourier-Laplace
domain problem and provide error estimates for the convolution quadrature and trigonometric
collocation scheme.  Finally in Section \ref{numerics} we provide some preliminary 2D numerical results, and draw some conclusions in Section \ref{concl}.

\section{The semi discrete problem}

We shall use the convolution quadrature approach \cite{lubich} to approximate the volume integral equation  (\ref{ust}) in time. A simple way to see how this arises is to start by discretizing the partial differential equation in time.  This is easier to understand if we temporarily
write the equation as a mixed system, defining $w=(u^s/c^2+q_cu^i/c_0^2)_t$. Then we obtain the  system
\begin{eqnarray}
\frac{1}{c^2}u^s_t+\frac{1}{c_0^2}q_cu^i_t&=&w\mbox{ in }\real^3 \times \real_+, \label{firsta}\\
w_t&=&\Delta u^s \mbox{ in }\real^3\times \real_+.\label{firstb}
\end{eqnarray}
 Now suppose we apply a multistep method to this problem.  To define the multistep method, suppose $y'=g(t,y)$, $t>0$ and $y(0)=0$. We write
$y_n\approx y(t_n)$ where $t_n=n\Delta t$, $n=0,1,\dots,$ and $\Delta t>0$ is the time step. Then we require $\{y_n\}$ to 
satisfy
\begin{equation}
\sum_{j=0}^k\alpha_jy_{n-j}=\Delta t\sum_{j=0}^k\beta_j g(t_{n-j},y_{n-j}),\quad n=0,1, \dots,
\label{multi}
\end{equation}
where $\{\alpha_j,\beta_j\}$ are constants describing the multistep method and we assume $\alpha_0/\beta_0>0$. We take $y_n=0$ if $n\leq 0$ because of the assumed zero initial condition at $t=0$.

Using this method on the first order system (\ref{firsta})-(\ref{firstb}), we compute $(w_n(x),u^s_n(x))$, $n=0,1,\dots,$ that satisfy
\begin{eqnarray*}
\sum_{j=0}^k\alpha_j\left(\frac{1}{c^2}u^s_{n-j}+\frac{q_c}{c_0^2}u^i_{n-j}\right)&=&\Delta t\sum_{j=0}^k\beta_j w_{n-j},\quad n=0,1,\dots,\\
\sum_{j=0}^k\alpha_jw_{n-j}&=&\Delta t\sum_{j=0}^k\beta_j \Delta u^s_{n-j},\quad n=0,1,\dots,
\end{eqnarray*}
where $u^i_{n}=u^i(x,t_n)$. Formally let $\cw:=\cw(x,\xi)=\sum_{j=0}^{\infty}w_j(x)\xi^j$, $\cu^s:=\cu^s(x,\xi)=\sum_{j=0}^{\infty}u^s_j\xi^j$, and  $\cu^i:=\cu^i(x,\xi)=\sum_{j=0}^{\infty}u^i_j\xi^j$, for $\xi\in\complex$.   Then multiplying the above system by $\xi^n$ and summing over $n$, and using the fact that the discrete fields are casual (i.e. $w_j(x)=0$ for $j\leq 0$ and similarly for the other fields),  we obtain
\begin{eqnarray*}
\left(\sum_{j=0}^k\alpha_j\xi^j\right) \left(\frac{1}{c^2}\cu^s+\frac{q_c}{c_0^2}\cu^i\right)&=&\Delta t\left(\sum_{j=0}^k\beta_j \xi^j\right) \cw,\\
\frac{1}{c^2}\left(\sum_{j=0}^k\alpha_j\xi^j\right)\cw&=&\Delta t\left(\sum_{j=0}^k\beta_j \xi^j\right)\Delta \cu^s. 
\end{eqnarray*}
Defining
\[
\delta(\xi)=\frac{\left(\sum_{j=0}^k\alpha_j\xi^j\right)}{\left(\sum_{j=0}^k\beta_j \xi^j\right)}
\]
and eliminating $\cw$ we obtain the following Fourier-Laplace domain Helmholtz equation for $\cu^s$:
\begin{equation}
-\Delta \cu^s+\frac{1}{c^2}\left(\frac{\delta(\xi)}{\Delta t}\right)^2\cu^s=-\left(\frac{\delta(\xi)}{\Delta t}\right)^2\frac{q_c}{c_0^2} \, \cu^i\mbox{ in }\real^3.
\label{delta_helm}
\end{equation}
Proceeding formally we can solve this problem using a volume integral equation.  To this end, we rewrite (\ref{delta_helm}) as
\begin{equation}
-\Delta \cu^s+\frac{1}{c_0^2}\left(\frac{\delta(\xi)}{\Delta t}\right)^2\cu^s=-\frac{1}{c_0^2}q_c\left(\frac{\delta(\xi)}{\Delta t}\right)^2(\cu^s
+\cu^i)\mbox{ in }\real^3.
\label{delta_helm_1}
\end{equation}
Now recalling (\ref{Vhatdef}) and (\ref{Vhatprop}) and   choosing  $s=\delta(\xi)/\Delta t\in\complex$ we see that $\cu^s\in L^2(D)$ satisfies
\begin{equation}
\cu^s+\frac{s^2}{c_0^2}\hat{V}(q_c\cu^s)=-\frac{s^2}{c_0^2}\hat{V}(q_c\cu^i)\mbox{ in }L^2(D).\label{usess}
\end{equation}
Comparing this equation with (\ref{ushateq}) suggests that time discretization of the integral equation can be understood
by analyzing (\ref{ushateq}) for suitable choices of $s$.  From now on $s$ will be a general transform parameter with
$s\in \complex$, $\Re(s)=\sigma>\sigma_0>0$ for some constant $\sigma_0$.

The Lippmann-Schwinger equation (\ref{usess}) with complex coefficient $s$ can  be solved for $\cu^s$ for any $\xi\in \complex$ small enough provided $\Re(s)>0$ (this will occur with the correct choice of the multi-step integrator).  The calculation of a time semi-discrete solution $u^{s,\Delta t}_n$, $n=0,1,\dots,$ from $\cu^s$ can be organized in one of two ways.  Classically we can determine 
a volume integral equation for $u^{s,\Delta t}_n$ in terms of previous values $u^{s,\Delta t}_j$, $j<n$  and hence arrive at a ``marching on in time'' scheme
for the integral equation determining $u^{s,\Delta t}_n$ successively \cite{lubich}.  Alternatively we can solve
for $\cu^s$ at suitable choices of the parameter $\xi$ and determine an approximation to $u_n^{s,\Delta t}$ by inverting the $\xi$ transform as
is done in \cite{Banjai_2008}.  The first method is necessary if the field is required at all points in the domain but requires  saving the solution at all time steps, whereas the second method requires the solution of many easy parallelized integral equations and works well if the field is only required at a small number of points in space.

\section{Notation and Preliminaries}\label{note}
To put the time and frequency domain integral equations  considered earlier on a firm footing, appropriate function spaces for the time domain solution of the wave equation are crucial to the analysis of the
method.  We now summarize briefly some suitable Sobolev spaces (for more details sees \cite{BHa86,lubich,had03}).  For $X$ a Hilbert space we denote by ${\cal D}(\real;X)=C_0^{\infty}(\real;X)$ the set of smooth and compactly supported $X$-valued functions.  Then ${\cal D}'(\real_+;X)$ are the $X$-valued distributions on the real line that vanish for times $t<0$, and the corresponding tempered distributions are ${\cal S}'(\real_+;X)$.  We can define
\[
{\cal L}'(\real_+;X)= \left\{f\in {\cal D}'(\real_+;X)\;|\;\exp(-\sigma_0t)f\in {\cal S}'(\real_+;X) \mbox{ for some }\sigma_0(f)<\infty \right\}.
\]
Functions $f\in{\cal L}'(\real_+,X)$ have a well defined Fourier-Laplace transform $\hat{f}$ given, as before, by
\[
\hat{f}=\hat{f}(s):=\int_{0}^\infty \exp(-st)f(t)\,dt\mbox{ for }\Re(s)>\sigma_0(f).
\]
If $\Re(s)=0$ then the Fourier-Laplace transform reduces to the Fourier transform on causal functions.  We introduce the Hilbert spaces
\[
H_\sigma^p(\real_+;X):=\left\{f\in {\cal L}'(\real_+;X)\:|\:\int_{\sigma+i\real}|s|^{2p}\Vert \hat{f}(s)\Vert_X^2\,ds<\infty\right\} \mbox{ for } p\in\N_0.
\]
Note that by using Parseval's theorem, the  norm on this space is equivalent to the time domain norm
\[
\Vert f\Vert_{H^p_\sigma(\real_+;X)}^2=\int_{0}^\infty \exp(-2\sigma t)\Vert \partial^p  f/\partial t^p(t)\Vert_X^2\;dt,
\]
where we have used the fact that $f(t)$ vanishes for $t<0$.

It will be convenient to define $L^2_{q_c}(D)$ via the weighted norm
\[
\Vert u\Vert_{L^2_{q_c}(D)}^2=\int_Dq_c|u|^2\,dV,
\]
so that $L^2_{q_c}(D)$ is the closure of $L^2(D)$ in the $ \Vert \cdot\Vert_{L^2_{q_c}(D)}$ norm.

We now prove estimates for $\hat{v} \mapsto \hat{V}(q_c \hat{v})$ (which also hold of course when $q_c=1$).

\begin{lemma}\label{little}
The operator $\hat{V}$ can be extended to an operator from $L^2(D)$ into $H^2(\real^3)$ and satisfies the following bounds for 
all $\hat{v}\in L^2(D)$, 
\begin{eqnarray}
\Vert \hat{V}(q_c\hat{v})\Vert_{L^2(D)}&\leq& \frac{q^{1/2}_{c,+}}{\sigma |s|}
\Vert \hat{v}\Vert_{L^2_{q_c}(D)},\label{VL2}\\
\Vert \hat{V}(q_c\hat{v}) \Vert_{H^1_s(\real^3)}&\leq &\frac{q^{1/2}_{c,+}}{\sigma}\Vert \hat{v}\Vert_{L_{q_c}^2(D)},\label{VL3}\\
\vert \hat{V}(q_c\hat{v})\vert_{H^2(\real^3)}&\leq &C\frac{|s|}{\sigma}\Vert \hat{v}\Vert_{L_{q_c}^2(D)},\label{VL4}
\end{eqnarray}
where $\Vert w\Vert_{H^1_s(\real^3)}^2=\Vert \nabla w\Vert_{L^2(\real^3)}^2+|s|^2\Vert w\Vert_{L^2(\real^3)}^2$, $\vert \cdot \vert_{H^2(\real^3)}$ is the $H^2(\real^3)$ semi-norm, and $\Re(s) = \sigma > 0$.
\end{lemma}
\begin{proof} The proof uses the techniques from \cite{BHa86}.  Let $\hat{w}=\hat{V}(q_c\hat{v})$ then $w\in H^1(\real^d)$ satisfies
\begin{equation}
-\Delta \hat{w}+s^2 \hat{w} = q_c\hat{v}\in \real^3.
\label{wv}
\end{equation}
Multiplying this equation by $\overline{sw}$ and integrating we obtain
\[
\overline{s}\Vert \nabla \hat{w}\Vert_{L^2(\real^3)}^2+|s|^2s\Vert \hat{w}\Vert_{L^2(\real^3)}^2=\overline{s}\int_D q_c\hat{v}\overline{w}\,dV.
\]
Taking real parts, we deduce that 
\begin{eqnarray*}
\sigma \big(\Vert \nabla \hat{w}\Vert_{L^2(\real^3)}^2+|s|^2\Vert \hat{w}\Vert_{L^2(\real^3)}^2 \big)
&\leq& |s| \Vert q_c^{1/2}\Vert_{L^\infty(D)} \Vert \hat{w}\Vert_{L^2(D)}\Vert q_c^{1/2}\hat{v}\Vert_{L^2(D)}\\
&\leq& \Vert q_c^{1/2}\Vert_{L^\infty(D)} \sqrt{\Vert\nabla \hat{w}\Vert_{L^2(D)}^2+|s|^2\Vert \hat{w}\Vert_{L^2(D)}^2}\Vert \hat{v}\Vert_{L_{q_c}^2(D)}
\end{eqnarray*}
and hence prove (\ref{VL3}).

To prove (\ref{VL2}), the above estimate also shows that
\[
\sigma |s|^2\Vert \hat{w}\Vert_{L^2(D)} \leq |s|\Vert q_c^{1/2}\Vert_{L^\infty(D)} \Vert \hat{v}\Vert_{L_{q_c}^2(D)}.
\]
Since (\ref{wv}) holds in $L^2(\real^3)$ we have
\[
\vert \hat{w}\vert^2_{H^2(\real^3)}\leq C \Vert \Delta \hat{w}\Vert_{L^2(\real^3)}^2\leq C(|s|^4\Vert \hat{w}\Vert^2_{L^2(\real^3)}+\Vert q_c\hat{v}\Vert_{L^2(\real^3)}^2).
\]
Using (\ref{VL2}), there is a constant $C$ depending on $q_c$ but independent of $s$ and $\hat{v}$ such that
\[
\vert \hat{w}\vert^2_{H^2(\real^3)}\leq C(|s|^2/\sigma^2+1)\Vert \hat{v}\Vert_{L_{q_c}^2(D)}^2.
\]
Noting that $|s|/\sigma\geq 1$, this proves (\ref{VL4}) and completes the proof.\end{proof}

Applying this lemma proves the following theorem:
\begin{theorem}\label{th:volPot}
The retarded volume potential operator $V$ can be extended to a bounded operator from $H^p_\sigma(\real_+;L^2(D))$ into $H_\sigma^{p+1-r}(\real_+;H^r(\real^3))$ for $r=0,1,2$.  Moreover, if $v=V(f)$ for some $f\in H^p_\sigma(\real_+;L^2(D))$ then
$v(t)=0$ for $t<0$ and $v\in H^p_\sigma(\real_+;H^1(\real^3))$ satisfies
\[
\frac{1}{c_0^2}v_{tt}-\Delta v=f\mbox{ in } H^{p-1}(\real_+;L^2(\real^3)).
\]
\end{theorem}

\begin{proof} Using  techniques from \cite{BHa86}, the  mapping properties follow from Lemma \ref{little} after inverting the Fourier-Laplace transform (and choosing $q_c=1$). The fact that $\hat{v}=\hat{V}(\hat{f})$ satisfies the Fourier-Laplace domain Helmholtz equation (\ref{wv}) shows that $v$ satisfies the wave equation in the claimed space.

To prove causality, we note that  this is clear for $f\in C^{\infty}(\real_+;C^{\infty}(D))$ due to the retarded potential representation
\[
(V(f))(t,x)=\int_D\frac{f(y,t-\Vert x-y|\Vert/c_0)}{4\pi\Vert x-y\Vert}\,dV
\]
since then $V(f)(\cdot,t)=0$ for $t<0$.
From the already established mapping properties, the density of smooth functions in $H^p_\sigma(\real_+;L^2(D))$ implies $(V(f))(x,t)=0$ for $t<0$ and 
$f\in H^p_\sigma(\real_+;L^2(D))$.
\end{proof}

\section{Existence and Operator Estimates}\label{laplace}
We now prove existence, uniqueness and semi-discrete error estimates for the solution of the time domain Lippmann-Schwinger problem using Lubich's  theory \cite{lubich}.   This result also underlies our later error estimates. The key to this approach is to show that the integral operator is coercive.
  To state the theorem we use the notation that if $A:X\to Y$ where $X$ and $Y$ are Hilbert spaces then the operator norm of $A$ is denoted
by $\Vert A\Vert_{X\to Y}$ where
\[
\Vert A\Vert_{X\to Y}=\sup_{v\in X}\frac{\Vert Av\Vert_Y}{\Vert v\Vert_X}.
\]
We also define $\hat{V}_{q_c}:L^2_{q_c}(D)\to L^2(D)$ by
\[
\hat{V}_{q_c}(g)=\hat{V}(q_cq)\mbox{ for all } g\in L^2_{q_c}(D).
\]

\begin{theorem} \label{solEst}
For any $s$ with $\Re(s)=\sigma>0$
\begin{equation} \label{solBound2}
\big\Vert \big(I+(s^2/c_0^2) \hat{V}_{q_c} \big)^{-1}(s^2/c_0^2)\hat{V}_{q_c} \big\Vert_{L_{q_c}^2(D)\to L_{q_c}^2(D)} \leq \frac{q_{c,+}}{\sigma^2c_0^2}|s|^2,
\end{equation}
and 
\begin{equation} \label{solBound}
\big\Vert \big(I+(s^2/c_0^2) \hat{V}_{q_c} \big)^{-1} \big\Vert_{L^2_{q_c}(D)\to L^2_{q_c}(D)} \leq  \frac{|s| }{\sigma}.
\end{equation}
In addition both operators above are analytic in $s$ for $\Re(s)>0$.
\end{theorem}

\begin{proof}
This proof extends the techniques of \cite{BHa86} to volume equations. Given $f\in L_{q_c}^2(D)$, consider the problem of finding $\hat{v}\in L_{q_c}^2(D)$ such that 
\begin{equation}
\hat{v}+\frac{s^2}{c_0^2} \hat{V}(q_c\hat{v})=-\frac{s^2}{c_0^2}\hat{V}(q_cf)\;\mbox{ in } L^2(D).\label{veq}
\end{equation}
We now derive a variational scheme for this problem and use the Lax-Milgram Lemma to
verify that it has a unique solution.  Multiplying (\ref{veq}) by the complex conjugate of $\xi q_c$ where $\xi\in L_{q_c}^2(D)$ (denoted $\overline{\xi}q_c$) and integrating over $D$ we obtain the problem of finding $\hat{v}\in L_{q_c}^2(D)$ such that
\begin{equation}
a(\hat{v},\xi)=b(\xi)\mbox{ for all } \xi\in L^2(D)\label{var}
\end{equation}
where
\[
a(\hat{v},\xi)=\int_D\left(q_c\hat{v}+\frac{s^2q_c}{c_0^2}\hat{V}(q_c\hat{v})\right)\overline{\xi}\,dV,
\]
and
\[
b(\xi)=-\frac{s^2}{c_0^2}\int_Dq_c\hat{V}(q_cf)\overline{\xi}\,dV.
\]
To show that the above problem has a unique solution we now verify the conditions of the Lax-Milgram Lemma.
To verify coercivity we choose $\xi=s\hat{v}$ so
\[
a(\hat{v},s\hat{v})=\int_D\overline{s}{q_c}|\hat{v}|^2\,dV+\frac{s|s|^2}{c_0^2}
\int_Dq_c\hat{V}(q_c\hat{v})\overline{ \hat{v}}\,dV.
\]
Taking real parts
\[
\Re(a(\hat{v},s\hat{v}))=\sigma\int_D{q_c}|\hat{v}|^2\,dV+\Re\left(\frac{s|s|^2}{c_0^2}
\int_D\hat{V}(q_c\hat{v})\overline{ q_c\hat{v}}\,dV\right).
\]
The second term on the right hand side is analyzed as follows.  Let $\hat{z}=\hat{V}(q_c\hat{v})$ then
\[
-\Delta \hat{z}+\frac{s^2}{c_0^2}\hat{z}=q_c\hat{v}\mbox{ in } \real^3,
\]
so that
\begin{eqnarray*}
\int_D\hat{V}(q_c\hat{v})\,q_c\overline{\hat{v}}\,dV
&=&\int_{\real^3}\hat{z}\overline{\left(-\Delta \hat{z}+\frac{s^2}{c_0^2}\hat{z}\right)}\,dV\\
&=&\int_{\real^3}|\nabla \hat{z}|^2+\frac{\overline{s}^2}{c_0^2}|\hat{z}|^2\,dV.
\end{eqnarray*}
This implies that
\begin{eqnarray*}
\Re\left(
\frac{s|s|^2}{c_0^2}\int_D\hat{V}(q_c\hat{v})\,q_c\overline{ \hat{v}}\,dV\right)
=\Re\left(\frac{s|s|^2}{c_0^2}\int_{\real^3}|\nabla w|^2\,dV+\frac{\overline{s}|s|^4}{c_0^4}\int_{\real^3}|w|^2\,dV\right)\geq 0.
\end{eqnarray*}
Thus we have shown that
\[
|a(\hat{v},s\hat{v})|\geq \sigma \left\Vert {q_c^{1/2}}\hat{v}\right\Vert_{L^2(D)}^2={\sigma}
\Vert \hat{v}\Vert_{L_{q_c}^2(D)}^2.
\]

We now need to show that $a(\cdot,s\cdot)$ is continuous. Clearly
\[
|a(\hat{v},s\xi)|\leq \left(|s|\Vert\hat{v}\Vert_{L_{q_c}^2(D)}+(q^{1/2}_{c,+}|s|^3/c_0^2)\Vert \hat{V}(q_c\hat{v})\Vert_{L^2(D)})\right)\Vert\xi\Vert_{L_{q_c}^2(D)}\leq C\Vert\hat{v}\Vert_{L_{q_c}^2(D)}\Vert\xi\Vert_{L_{q_c}^2(D)}
\]
for some constant $C=C(|s|)$ where we have 
 estimated $\Vert \hat{V}(q_c\hat{v})\Vert_{L^2(D)}$ using Lemma \ref{little}. 

We also need to show that $b(s\cdot)$ is an anti-linear functional and estimate it. Antilinearity is obvious, and
\[
|b(s\xi)| =\frac{|s|^2}{c_0^2}\left|\int_D\hat{V}(q_cf)\,q_c\overline{s\xi}\,dV\right|\leq \frac{q_{c,+}^{1/2}|s|^3}{c_0^2}\Vert\hat{V}(q_cf)\Vert_{L^2(D)}\Vert q_c^{1/2}{\xi}\Vert_{L^2(D)}.
\]
Using Lemma \ref{little} to estimate $\hat{V}(q_cf)$ gives
\[
|b(s\xi)| \leq q_{c,+}\frac{|s|^2}{\sigma c_0^2}\Vert f \Vert_{L_{q_c}^2(D)}\Vert{\xi}\Vert_{L_{q_c}^2(D)}.
\]
The conditions of the Lax-Milgram lemma are now satisfied and we may conclude that
the problem of finding $\hat{v}\in L_{q_c}^2(D)$ such that
\[
a(\hat{v},s\xi)=b(s\xi)\mbox{ for all } \xi\in L_{q_c}^2(D)
\]
has a unique solution depending continuously on the data.  Furthermore choosing $\xi=\hat{v}$ gives
\[
\sigma\Vert \hat{v}\Vert_{L_{q_c}^2(D)}\leq q_{c,+}\frac{|s|^2}{\sigma c_0^2}\Vert f \Vert_{L_{q_c}^2(D)}.
\] 

For the second estimate of the theorem, we redefine
\[
b(\xi)=\int_Dq_c\hat{f}\xi\,dV
\]
for some $\hat{f} \in L^2(D)$.  Then
\[
|b(s\xi)|\leq |s|\Vert\xi\Vert_{L_{q_c}^2(D)}\Vert\hat{f}\Vert_{L_{q_c}^2(D)},
\]
and proceeding as before we obtain the second estimate.
\end{proof}

\begin{corollary}
  If $\hat{u}^i \in L_{q_c}^2(D)$, then the unique solution $\hat{u} \in L_{q_c}^2(D)$ to
  \begin{equation} \label{eq:hatUEq}
    \hat{u} + \frac{s^2}{c_0^2} \hat{V}(q_c \hat{u} ) = \frac{s^2}{c_0^2} \hat{V}(q_c \hat{u}^i) 
    \quad \text{in } L_{q_c}^2(D)
  \end{equation}
  belongs to $H^1_s(\real^3)$ and satisfies the following estimates for $\Re(s)>\sigma_0>0$,
  \begin{align}
  \|\hat{u} \|_{L_{q_c}^2(D)} 
    & \leq \frac{q_{c,+}}{\sigma^2 c_0^2}|s|^2 \| \hat{u}^i \|_{L_{q_c}^2(D)},\\
  \|\hat{u} \|_{H^1_s(\real^3)} 
    & \leq \frac{|s|^2}{\sigma c_0^2} q_{c,+}^{1/2} \left(1+ \frac{q_{c,+}}{\sigma^2 c_0^2}|s|^2 \right) \| \hat{u}^i \|_{L_{q_c}^2(D)},\\
  \| \hat{u} \|_{H^2(D)} 
    & \leq  C |s|^5 \| \hat{u}^i \|_{L_{q_c}^2(D)}, \label{eq:aux11}  
  \end{align}
  where $C$ depends on $\sigma_0$.
\end{corollary}
\begin{proof}
Combining Lemma~\ref{little} with Theorem~\ref{solEst} shows that 
\begin{align*}
  \| \hat{u} \|_{H^1_s(\real^3)}
  & = \frac{|s|^2}{c_0^2} \| \hat{V}(q_c(\hat{u} + \hat{u}^i)) \|_{H^1_s(\real^3)} \\
  & \leq \frac{|s|^2}{c_0^2} \frac{q_{c,+}^{1/2}}{\sigma_0} \| \hat{u} + \hat{u}^i \|_{L^2_{q_c}(D)}
  \leq \frac{|s|^2}{c_0^2} \frac{q_{c,+}^{1/2}}{\sigma_0} \left( 1 + \frac{q_{c,+}}{\sigma^2 c_0^2}|s|^2 \right)\| \hat{u}^i \|_{L^2_{q_c}(D)}. 
\end{align*}
Further, using (\ref{VL4})
\begin{eqnarray*}
  |u|_{H^2(D)} &
  =& \frac{|s|^2}{c_0^2}  | \hat{V}(q_c(\hat{u} + \hat{u}^i))|_{H^2(\R^3)}
  \leq C \frac{|s|^3}{\sigma} \| \hat{u} + \hat{u}^i \|_{L_{q_c}^2(D)}\\
&  \leq &C |s|^5 \| \hat{u}^i \|_{L_{q_c}^2(D)}.
\end{eqnarray*}
\end{proof}

Using Lubich's theory, the first estimate in Theorem~\ref{solEst} shows that semi-discretization in time via convolution quadrature results in
an optimally convergent semi-discrete method (discrete in time) with no time step restrictions. To state this
result we adopt Lubich's notation. Define $\hat{A}(s): \, L_{q_c}^2(D)\to L_{q_c}^2(D)$ by 
\[
\hat{A}(s)=-\left(I+(s^2/c_0^2) \hat{V}_{q_c} \right)^{-1}(s^2/c_0^2)\hat{V}_{q_c}. 
\]
Then denote the corresponding time domain solution operator (obtained by the inverse Fourier-Laplace transform) by $\hat{A}(\partial_t)$, note that Lemma~\ref{little} and Theorem~\ref{solEst} state that  
\[
\hat{A}(\partial_t):\,H^{p+2}_\sigma(\real_+;L_{q_c}^2(D))\to H^{p}_\sigma(\real_+;L_{q_c}^2(D))
\]
is bounded.  Then we denote by $\hat{A}(\partial_t^{\Delta t})$ the convolution quadrature semi-discrete in time
solution operator.  Using this notation, the exact solution $u^s$ is given by
\[
u^s=\hat{A}(\partial_t)(u^i)
\]
and at each time step the semi-discrete solution denoted $u^{s,\Delta t}$ is given by
\[
u_n^{s,\Delta t}=\hat{A}(\partial_t^{\Delta t})(u^i)(n\Delta t), \  n=0,1,\dots \,.
\]
The following estimate holds up to a fixed final time $T>0$, where now $\Delta t=T/M$ for some $M>0$.
\begin{theorem} \label{semid} Suppose the multistep method (\ref{multi}) is A-stable and of order $p$, and $\delta(\zeta)$ has no poles on the unit circle in the complex plane.  Let $u^i\in H^r_\sigma(\real_+,L_{q_c}^2(D))$ with $r=p+3$ then
\[
\left(\Delta t \sum_{n=0}^M\Vert u^s(t_n)-u^{s,\Delta t}_n\Vert_{L_{q_c}^2(D)}^2\right)^{1/2}\leq C (\Delta t)^p\Vert u^{i}\Vert_{H^r_{\sigma_0}(\real_+,L_{q_c}^2(D))}.
\]
Here $C$ is independent of $\Delta t$ and $u^i$ but depends on $p$ and $\sigma_0$.
\end{theorem}

\begin{remark} \begin{enumerate} \item Since Theorem \ref{solEst} holds for conforming Galerkin methods based on (\ref{var}) and conforming finite element subspaces of $L_{q_c}^2(D)$ we could also prove a fully discrete error estimate without stability restrictions
(c.f. Theorem 5.4 of \cite{lubich}).  However to simplify implementation of the fully discrete scheme and provide fast operator evaluation, we  instead will follow a different approach.  We will analyze a collocation scheme based on periodization of the integral equation which diagonalizes the integral operator $\hat{V}$.  This implies that $\hat{V}$ can be applied rapidly but also introduces an extra spatial approximation.

\item Estimates in other norms could also be proved (see \cite{lubich}).
\end{enumerate}\end{remark}

\begin{proof} We apply Theorem 3.3 of \cite{lubich}.\end{proof}

\section{Collocation Discretization for the Lippmann-Schwinger Equation at Complex Frequency}\label{trig}
To obtain a fast solver for the Fourier-Laplace domain Lippmann-Schwinger integral equation~\eqref{ushateq} we use a trigonometric collocation approach as in \cite{vain00}. This involves periodization of the volume integral operator $\hat{V}$ that appears in~\eqref{ushateq}. Since we want to combine this collocation discretization in space with convolution quadrature in time, all estimates for the spatial discretization have to be explicit in terms of the complex frequency $s$.  

Extending the contrast function $q_c$ by zero to all of $\real^3$, we assume that 
\begin{equation} \label{eq:supAss}
  \overline{D} = \mathrm{supp}\,(q_c ) \subset \{ x \in \real^3, \, |x| \leq \rho \} \quad \text{for some }\rho > 0  
\end{equation}
and consider again the frequency-domain integral equation from~\eqref{ushateq},  
\begin{equation} \label{IEFreq}
  \hat{u}^s+ \frac{s^2}{c_0^2} \hat{V}(q_c \hat{u}^s) = -\frac{s^2}{c_0^2}\hat{V}(q_c\hat{u}^i) 
  \quad \text{in } L^2(D). 
\end{equation}
Due to~\eqref{eq:supAss}, we can extend the restriction of $q_c$ to 
\[
  G_{2\rho} = \{ x \in \real^3, \, -2\rho < x_j \leq 2\rho , \, j=1,2,3\}, 
\]
as a $4\rho$-periodic function in each direction in space that we call $q_{c,\p}$. The extension $q_{c,\p}$ obviously is as smooth as $q_c$. Below, we interpret $q_{c,\p}$ as a function in a periodic function space defined on $G_{2\rho}$. We will also periodize the integral operator  $\hat{V}$ from (\ref{Vhatdef}). To this end, we define a $4\rho$-periodic kernel by extending the kernel function 
\[
  \kappa_{\p,s}(x) = 
  \begin{cases}
    \frac{\exp(-s/c_0 \, |x|)}{4\pi |x|} & x \in G_{2\rho}, \, 0< |x| < 2\rho, \\
    0 & x \in G_{2\rho}, \, 2\rho \leq |x|, 
  \end{cases}
\]
$4\rho$-periodically in each coordinate direction of space to a function on $\real^3$ defined almost everywhere. 
The associated $4\rho$-periodic integral operator is
\[
  \hat{V}_\p(f) = \int_{G_{2\rho}} \kappa_{\p,s}(\cdot - y) f(y) \, dV(y).
\]
Later on, we will prove mapping properties of $\hat{V}_\p$ in periodic Sobolev spaces. We further replace $\hat{u}^i$ on the right-hand side of~\eqref{IEFreq} by its extension by zero to $G_{2\rho}$ denoted $\hat{u}^i_{\p} \in L^2(G_{2,\rho})$. As in~\cite{vain00, Saran2002} we note that~\eqref{IEFreq} is equivalent to the following periodic integral equation for the unknown $\hat{u}_\p \in L^2(G_{2\rho})$,
\begin{equation} \label{IEFreqPer}
  \hat{u}_\p + \frac{s^2}{c_0^2} \hat{V}_\p(q_{c,\p} \hat{u}_\p) = \hat{V}_\p(q_{c,\p} \hat{u}^i_\p) 
  \qquad \text{in } L^2(G_{2\rho}). 
\end{equation}
Note that $q_{c,\p} \hat{u}^i_\p$ vanishes outside $D$; thus, as in~\eqref{IEFreq}, the values of $\hat{u}^i_\p$ in $G_{2\rho} \setminus \overline{D}$ play no role. 
The equivalence of~\eqref{IEFreqPer} and~\eqref{IEFreq} is due to the fact that for $x, y \in D$ it holds that $|x-y| \leq |x| + |y| < 2\rho$, thus $\kappa_{\p,s}(x - y) = \exp(-s/c_0 \, |x-y|) / (4\pi |x-y|)$ for $x \not = y$ and
\begin{eqnarray} 
  \hat{V}_\p(q_{c,\p} f)
  &=& \int_{G_{2\rho}} \kappa_{\p,s}(\cdot\, - y) q_{c,\p}(y) f(y) \, dV
  = \int_{D} \frac{e^{-s/c_0 \, |\cdot-y|}}{4\pi |\,\cdot \, -y|} q_{c,\p}(y) f(y) \, dV \nonumber\\
  &=& \hat{V} ( q_c f ), \quad x \in D. \label{eq:aux6}
\end{eqnarray}
Hence, if $\hat{v}$ solves~\eqref{IEFreq}, then $\hat{V}_\p(q_{c,\p} (\hat{u}^i - (s^2/c_0^2) \, \hat{v}_\p ))$ defines a solution to the periodic integral equation, and if $\hat{v}_\p$ solves the periodic equation, then $\hat{V}(q_c (\hat{u}^i - (s^2/c_0^2) \, \hat{v}_\p ))$ defines a solution to~\eqref{IEFreq}. (By abuse of notation, we did not explicitly write down the necessary extensions by zero from $D$ to $G_{2\rho}$ and restrictions from $G_{2\rho}$ to $D$.) 
%

Next, we recall well known facts about trigonometric functions and associated spaces and operators. All concepts are explained in more detail in, e.g.,~\cite{Saran2002}, together with proofs for the corresponding one- and two-dimensional results (the proofs in three dimensions are analogous with obvious modifications due to the additional spatial dimension).
Using trigonometric monomials
\[
  \varphi_j(x) := \frac{1}{(4\rho)^{-3/2}} e^{\frac{\i\pi}{2\rho} \, j \cdot x}, \qquad j \in \Z^3, x \in \R^3, 
\]
we define, for $N \in \N$, the finite-dimensional subspace of trigonometric polynomials 
\[
  T_N := \mathrm{span} \big\{ \varphi_j,  j\in \Z^3_N \big\} \subset L^2(G_{2\rho}) 
  \]
  where \[
  \Z^3_N := \{ \ell\in\Z^3,  -N/2 \leq \ell_j < N/2,  j=1,2,3 \}. 
\]
Further, recall the definition of the Fourier coefficients of a $4\rho$-periodic distribution $v \in \mathcal{D}'_\p(\R^3)$, 
\[
  \dhat{v}(j) := (v, \varphi_{-j})_{\mathcal{D}'_\p \times \mathcal{D}_\p} 
  \quad 
  \left[ = \int_{G_{2\rho}} v(x) \overline{\varphi_j(x)} \, dV \text{ if } v \in L^2(G_{2\rho}) \right],  
\]
and the periodic Sobolev spaces $H^{t}$ for $t\in\R$, defined by 
\[
  H^{t} := \bigg\{ v \in \mathcal{D}'_\p, \, \sum_{j\in\Z^3} (1+|j|^2)^{t} \,  | \dhat{v}(j) |^2 <\infty \bigg\}
  \quad \text{with norm}\quad
  \| v \|_{H^t}^2 = \sum_{j\in\Z^3} (1+|j|^2)^{t} \,  | \dhat{v}(j) |^2.
\]
We recall that functions in $H^t$ with $t>3/2$ are continuous, due the Sobolev embedding theorem, and that $H^r$ is a Banach algebra for $t > 3/2$, that is, $\| uv\|_{H^t} \leq C(t) \| u\|_{H^t} \| v\|_{H^t}$. 

Following~\cite{Saran2002}, we also define $Q_N$ to be the interpolation operator that maps continuous, $4\rho$-periodic functions in $\R^3$ to their interpolation polynomial in $T_N$, defined by $Q_N(v) = v_N$ where $v_N \in T_N$ satisfies $v_N(x_j) = v(x_j)$ for the interpolation points  
\[
  x_j := jh, \quad j\in \Z^3_N , \quad h = \frac{4\rho}{N}. 
\]
As $N\to\infty$, the interpolation projection $Q_N v$ approximates $v \in H^t$ if $t>3/2$, since the error estimate 
\begin{equation} \label{eq:errQ}
  \| [I- Q_N] f\|_{H^r} \leq C(\ell,t) N^{r-t}\| f \|_{H^t}
\end{equation}
holds for for $0 \leq r \leq t$ whenever $f \in H^t$ with $t>3/2$. 

Note that a function $v_N \in T_N$ can either be characterized by its Fourier coefficients 
\[
  \dhat{v}_N = \{ \dhat{v}(j) \}_{j\in\Z^3_N} \in \C^3_N = \{ c(j) \in \C, \, j \in \Z^3_N \}
\]
or its values at the nodal points $x_j$, that we abbreviate as $\underline{v}_N  = \{ v_N(x_j)\}_{j\in \Z^3_N} \in \C^3_N$. (By abuse of notation, we use the same notation for point evaluation at the nodal points of any continuous function.) It is well known that the three-dimensional discrete Fourier transform $\mathcal{F}_N$ is an isometry on $\C^3_N$ mapping the point values $\underline{v}_N \in \C^3_N$ of $v_N \in T_N$ to the Fourier coefficients $\dhat{v}_N \in \C^3_N$. 

Spaces of trigonometric polynomials are attractive for discretizing the integral operator $\hat{V}_\p$, since this periodic integral operator diagonalizes on trigonometric monomials, 
\begin{align*}
  (\hat{V}_\p (\varphi_j))(x) 
  & = \frac{1}{(4\rho)^{3/2}} \int_{G_{2\rho}} \kappa_{s,\p}(x-y)  e^{\frac{\i\pi}{2\rho} \, j \cdot (y-x)} dV(y) \, e^{\frac{\i\pi}{2\rho} \, j \cdot x} \\
  & = \int_{G_{2\rho}} \kappa_{s,\p}(z)  e^{- \frac{\i\pi}{2\rho} \, j \cdot z} dV(z) \, \varphi_j(x) 
  = (4\rho)^{3/2} \dhat\kappa_{s,\p}(j)  \varphi_j(x), \quad j \in \Z^3, x \in G_{2\rho}.
\end{align*}
Setting $R := 2\rho \i s / c_0$ for $s = \sigma + \i \eta$, the Fourier coefficients $\dhat\kappa_{s,\p}(j)$ of the periodic kernel $\kappa_{s,\p}$ can be explicitly computed as 
\begin{align}
  & \dhat\kappa_{s,\p}(j) 
  = \begin{cases}
    \frac{R^2}{\pi^2 |j|^2 - R^2} \left[ 1 - \exp(\i R) \left( \cos(\pi |j|)-\frac{\i R}{\pi |j|} \sin(\pi |j|) \right) \right] & j \not = 0, \\
    \exp(\i R) (1-\i R)-1 & j=0,
  \end{cases}
  \label{keFoCo1}  \\
  & = 
  \begin{cases}
    \frac{4 \rho^2}{c_0^2} \frac{\eta^2-\sigma^2-2\i \eta\sigma}{\pi^2 |j|^2 - 4 \rho^2(\eta^2-\sigma^2-2\i \eta\sigma)/c_0^2} \left[ 1 - e^{-\frac{2\rho}{c_0}(\sigma+\i \eta)} \left( \cos(\pi |j|)+\frac{2\rho (\sigma+\i \eta)}{c_0\pi |j|} \sin(\pi |j|) \right) \right] & j \not = 0, \\
    \exp[-2\rho(\sigma+\i \eta)/c_0] (1+2\rho(\sigma+\i \eta)/c_0)-1 & j=0.
  \end{cases}
  \label{keFoCo} 
\end{align}
Since we will always assume that $\Re(s) = \sigma>0$, the $j$th Fourier coefficient $\dhat\kappa_{\p,s}(j)$ is always well defined (the denominator in~\eqref{keFoCo} cannot vanish) and from~\eqref{keFoCo1} one observes that each of the coefficients is analytic in $s$ as long as $\Re(s)=\sigma>0$. We next establish estimates for the Fourier coefficients $\dhat\kappa_{s,\p}$ that will allow us to prove mapping properties of $\hat{V}_\p$ in the sequel. Writing again $s = \sigma + \i \eta$ and assuming that $\sigma > \sigma_0>0$, we find that 
$|\dhat\kappa_{s,\p}(0)| \leq (1+\exp[-2\rho\sigma_0/c_0] [1+2\rho|s|/c_0 ] )$ and 
\begin{align*}
 |\dhat\kappa_{s,\p}(j)| 
  & = 4 \rho^2 \left| \frac{\eta^2-\sigma^2-2\i \eta\sigma}{c_0^2 \pi^2 |j|^2 - 4 \rho^2(\eta^2-\sigma^2-2\i \eta\sigma)} \right| \\&\qquad\qquad\left| 1 - e^{-\frac{2\rho}{c_0}(\sigma+\i \eta)} \left( \cos(\pi |j|)+\frac{2\rho (\sigma+\i \eta)}{c_0\pi |j|} \sin(\pi |j|) \right) \right| \\
  & \leq \left| \frac{|s|^2}{c_0^2 \pi^2 |j|^2 /4 \rho^2 - (\eta^2-\sigma^2-2\i \eta\sigma)} \right| \left[ 1 +  e^{-\frac{2\rho\sigma}{c_0}} + e^{-\frac{2\rho\sigma}{c_0}} \frac{2\rho \sigma}{c_0\pi |j|} + e^{-\frac{2\rho\sigma}{c_0}} \frac{2\rho |\eta|}{c_0\pi |j|}\right]\\
  & \leq \left| \frac{|s|^2}{c_0^2 \pi^2 |j|^2 /4 \rho^2 - (\eta^2-\sigma_0^2-2\i \eta\sigma_0)} \right| \left[ 2 + \frac{1}{e \pi |j|} + e^{-\frac{2\rho\sigma_0}{c_0}} \frac{2\rho |\eta|}{c_0\pi |j|}\right], \qquad j \not = 0.
\end{align*}
Since $\eta \mapsto (\eta^2-\sigma_0^2-2\i \eta\sigma_0)$ describes a parabola in the complex plane intersecting the real axis orthogonally at $-\sigma_0^2$, the complex axis at $\pm 2 \i \sigma_0^2$ and with real part  tending to $+\infty$ as $\eta\to \pm\infty$, the distance between this parabola and the point $c_0^2 \pi^2 |j|^2 /4 \rho^2 >0$ on the positive real axis is strictly positive, 
$| c_0^2 \pi^2 |j|^2 /4 \rho^2 - (\eta^2-\sigma^2-2\i \eta\sigma)| \geq \sigma_0^2>0$, that is, $|\dhat\kappa_{s,\p}(j)| \leq C |s|^2|\eta| \leq C |s|^3$.
The following lemma sharpens this estimate. 

\begin{lemma}\label{lemma7}
For $s \in \C$ with $\Re s = \sigma > \sigma_0>0$ there is $C=C(\sigma_0) > 0$ independent of $s$ with
\begin{equation} 
  \label{eq:boundKappa}
  0< |\dhat\kappa_{s,\p}(j)| 
  \leq C \frac{|s|^2}{1+|j|} \quad \text{for all } j\in \Z^3.
\end{equation}
In consequence, $\hat{V}_\p$ is bounded from $H^t$ into $H^{t+1}$ for all $t\in\R$ and 
$\| \hat{V}_\p \|_{H^t \to H^{t+1}} \leq C |s|^2$ with $C$ independent of $t\in\R$. 
The integral operator $\hat{V}_\p$ is analytic in $s$ for all $s \in \C$ with 
$\Re s = \sigma > \sigma_0>0$.  
\end{lemma}
\begin{remark}\label{remark1}
  If one fixes $s$ or restricts $s$ to any bounded subset $M$ of $\{ s \in \C, \, \Re(s) > \sigma_0 > 0\}$, 
  then $|\dhat\kappa_{s,\p}(j)| \leq C_M /(1+|j|^2)$.
\end{remark}
\begin{proof}
For $|j|>0$ one computes that the minimal squared distance between the point $c_0^2 \pi^2 |j|^2 /4 \rho^2$ 
and the curve $\eta \mapsto \eta^2-\sigma^2-2\i \eta\sigma$ in the complex plane is attained for 
$\eta^2 = c_0^2 \pi^2 |j|^2 /4 \rho^2 - \sigma^2$ and equals $(\sigma  c_0 \pi \rho |j|)^2$ for $\sigma>0$. 
Moreover, 
$| c_0^2 \pi^2 |j|^2 /4 \rho^2 - (\eta^2-\sigma^2-2\i \eta\sigma) | \geq |\Im [\dots]| = 2 |\eta|\sigma$. 
Hence,
\[
  \left| \frac{|s|^2}{c_0^2 \pi^2 |j|^2 /4 \rho^2 - (\eta^2-\sigma^2-2\i \eta\sigma)} \right|
  \leq C \frac{|s|^2}{(1 + |j|)(1+|\eta|)}
\]
and 
\begin{align*}
  |\dhat\kappa_{s,\p}(j)| 
  & \leq C \frac{|s|^2}{1 + |j|} \left[ 2 + \frac{1}{e \pi |j|} + \frac{2\rho}{c_0\pi |j|} \frac{|\eta|}{1+|\eta|}\right] 
\end{align*}
for $j \not = 0$. For $j=0$ the claimed estimate is obvious. 
\end{proof}

Now we are ready to introduce the collocation discretization for the periodized integral equation~\eqref{IEFreqPer}. 
Assuming that $q_{c,\p} \in H^r$ for some $r>3/2$, we seek $\hat{u}_{\p,N} \in T_N$ that solves
\begin{equation} \label{IEColl}
  \hat{u}_{\p,N} + \frac{s^2}{c_0^2} \hat{V}_\p (Q_N[q_{c,\p}\hat{u}_{\p,N}])
  = \frac{s^2}{c_0^2} \hat{V}_\p (Q_N(q_{c,\p} \hat{u}^i_\p))
  \qquad \text{in } T_N. 
\end{equation}
Since $\hat{V}_\p$ diagonalizes on the trigonometric space $T_N$, an equivalent fully discrete formulation in terms of the Fourier coefficients $\dhat{u}_N$ of $\hat{u}_{\p,N}$ reads
\[
  \dhat{u}_{\p,N} + \frac{s^2}{c_0^2} \dhat\kappa_{s,\p} \bullet \mathcal{F}_N \left[ \underline{q_{c,\p}}_N \bullet \mathcal{F}_N^{-1} \dhat{u}_{\p,N} \right]
  = \frac{s^2}{c_0^2} \dhat\kappa_{\p,s} \bullet \mathcal{F}_N \left[ \underline{q_{c,\p}}_N \bullet \underline{\hat{u}^i_\p}_N \right] 
  \quad \text{in } \C^3_N
\]
where $\bullet$ denotes the element-wise multiplication of two vectors in $\C^3_N$. 
Obviously, this collocation discretization requires that the smallest $N^3$ Fourier coefficients of the discrete solution satisfy the Fourier transformed continuous equation rather than that the discrete solution itself satisfies the continuous equation at the collocation points. 

Due to the smoothing properties of $\hat{V}$ the discretized operator is close to the original one.

\begin{lemma} \label{th:OpApprox}
If $q_{c,\p} \in H^r$ with $r>3/2$, then there is $N_0 \in \N$ and $C>0$ independent of $s$ with $\Re(s) > \sigma_0>0$ such that  
\[
  \| \hat{V}_\p Q_N f - \hat{V}_\p Q_N  (q_{c,\p}f) \|_{H^2} 
  \leq C |s|^2 N^{-\min(1,r-1)} \| q_{c,\p} \|_{H^r} \| f \|_{H^2} 
  \quad \text{for all } f\in H^2. 
\]
\end{lemma}
\begin{proof}
We exploit the error estimate~\eqref{eq:errQ} for the interpolation projection and the fact  that $H^r$ is a Banach algebra for $r > 3/2$ to estimate  
\begin{align*}
  \left\| \hat{V}_\p[ q_{c,\p} f ] - \hat{V}_\p Q_N  (q_{c,\p}f) \right\|_{H^2}
  & \leq C |s|^2 \big\| [I-Q_N] (q_{c,\p}f) \big\|_{H^1} \\
 & \leq C |s|^2 N^{-\min(1,r-1)} \| q_{c,\p}f \|_{H^{\min(2,r)}} \\
  & \leq C |s|^2 N^{-\min(1,r-1)} \| q_{c,\p} \|_{H^{\min(2,r)}} \| f \|_{H^{\min(2,r)}}\\
  & \leq C |s|^2 N^{-\min(1,r-1)} \| q_{c,\p} \|_{H^r} \| f \|_{H^2}.  
\end{align*}
\end{proof}

\begin{lemma}\label{difficult}
If $q_{c,\p} \hat{u}^i_\p \in H^r$ with $r>3/2$ and $\Re(s) > \sigma_0 >0$, then there is $N_0=N_0(|s|) \in \N$  such that for $N\geq N_0$ there is a unique solution $\hat{u}_{\p,N} \in T_N$ to~\eqref{IEColl}. 
If $q_{c,\p} \hat{u}^i_\p \in H^{1+k}$ with $k\in\N=\{1,2,\dots \}$ and $q_{c,\p} \in C^{2}(G_{2\rho}) \cap H^{1+k}$, then there is $C=C(q_c,\sigma_0,k)>0$ independent of $s$ with $\Re(s) > \sigma_0 >0$ such that
\begin{equation} \label{eq:errEstSpace}
  \| \hat{u}_{\p,N} - \hat{u}_\p \|_{H^2}
  \leq C |s|^{18+4(k-1)} N^{-k} \left[ \| q_{c,\p} \hat{u}^i_\p \|_{H^{1+k}} 
  + \| q_{c,\p} \|_{H^{1+k}}^{k-1} \| q_c \|_{C^2(G_{2\rho})} \| \hat{u}^i \|_{L^2_{q_c}(D)} \right]
\end{equation}
for $N \geq N_0$, where $\hat{u}_\p \in H^2$ solves the periodized Lippmann-Schwinger equation~\eqref{IEFreqPer}. 
A lower bound for $N_0(|s|)$ is $c|s|^{13}$ with $c=c(q,\sigma_0)$. 
\end{lemma}
\begin{remark} 
  We note that the power of $|s|$ in~\eqref{eq:errEstSpace} reduces by two
  if one restricts $s$ to a bounded subset of $\{ s \in \C, \, \Re(s) > \sigma_0 > 0\}$ since the the remark after Lemma \ref{lemma7} indicates that estimate (\ref{eq:boundKappa}) is independent of $|s|$.
\end{remark}
\begin{proof}
Clearly, Lemma~\ref{th:OpApprox} implies that 
\begin{equation}\label{eq:aux1}
  \big\| \big[I + \frac{s^2}{c_0} \hat{V}_\p( q_{c,\p} \, \cdot ) \big] - \big[I + \frac{s^2}{c_0} \hat{V}_\p (Q_N (q_{c,\p} \, \cdot)) \big] \big\|_{H^2 \to H^2} 
  \leq C_A |s|^4 N^{-\min(1,1-r)} \| q_{c,\p} \|_{H^r}. 
\end{equation}
Recall that the continuous, non-periodized problem to find a solution $\hat{u} \in L_{q_c}^2(D)$ to the homogeneous equation~\eqref{IEFreq}, that is, $\hat{u} + (s^2/c_0^2) \, \hat{V}(q_{c,\p} \hat{u}) = 0$, possesses only the trivial solution due to Theorem~\ref{solEst}. Since any solution to the periodized problem~\eqref{IEFreqPer} yields a solution to~\eqref{IEFreq}, a solution to the homogeneous equation $\hat{u}_\p + (s^2 / c_0^2) \, \hat{V}(q_{c,\p} \hat{u}_\p) = 0$ must necessarily vanish in $D$, which directly implies that $\hat{u}_\p$ vanishes in $G_{2\rho}$. Hence, also equation~\eqref{IEFreqPer} possesses only the trivial solution for zero right-hand side. 

The operator $\hat{u}_\p \mapsto \hat{V}_\p(q_{c,\p} \hat{u})$ is bounded from $H^2$ into $H^{\min(2,r)+1}$ because multiplication by $q_{c,\p}$ is continuous from $H^2$ into $H^{\min(2,r)}$ and $\hat{V}_\p$ is bounded from $H^{\min(2,r)}$ into $H^{\min(2,r)+1}$. Since $\min(2,r)+1 > 5/2$, the embedding of the latter space into $H^2$ is compact, that is, $\hat{u}_\p \mapsto \hat{V}_\p(q_{c,\p} \hat{u}_\p)$ is compact on $H^2$. Thus, Fredholm's alternative implies that $I + (s^2 / c_0^2) \, \hat{V}_\p(q_{c,\p} \, \cdot)$ is invertible on $H^2$. 
By~\eqref{eq:aux1} and a Neumann series argument, this implies that $I + (s^2/c_0^2) \, \hat{V}_\p (Q_N (q_{c,\p} \, \cdot))$ is also invertible on $H^2$ if $N=N(s)$ is large enough, more precisely, if 
\begin{equation} \label{eq:cruxCo}
  C_A |s|^4 N^{-\min(1,1-r)} \| q_{c,\p} \|_{H^r} \bigg\| \bigg[ I + \frac{s^2}{c_0^2} \hat{V}_\p \big( q_{c,\p} \, \cdot \big) \bigg]^{-1} \bigg\|_{H^2 \to H^2}< 1
\end{equation}
To obtain an estimate for the operator norm of the inverse in the last condition, we suppose from now on that $q_{c,\p} \in C^2(G_{2\rho})$ (in particular, $r\geq 2$). 
Assume that $\hat{v}_\p \in H^2$ satisfies 
\begin{equation}
  \label{eq:aux3}
  \hat{v}_\p + \frac{s^2}{c_0^2} \hat{V}_\p \big[ q_{c,\p} \hat{v}_\p \big] = f \quad \text{in }H^2
  \qquad \text{for some $f\in H^2$.}
\end{equation}
Then using Lemma \ref{lemma7}
\begin{equation} \label{eq:aux13}
  \| \hat{v}_\p \|_{H^2} \leq \| f \|_{H^2} + \frac{|s|^2}{c_0^2} \| \hat{V}_\p [ q_{c,\p} \hat{v}_\p ] \|_{H^2} 
  \leq \| f \|_{H^2} + C |s|^4 \| q_{c,\p} \hat{v}_\p \|_{H^1}, 
\end{equation}
leading to 
\begin{align}
  \| \hat{v}_\p \|_{H^2} 
  & \leq \| f \|_{H^2} + C |s|^4 \| q_{c,\p} \|_{C^1(G_{2\rho})} \| \hat{v}_\p  \|_{H^1} \nonumber\\
  & \leq \| f \|_{H^2}+C |s|^6 \| q_{c,\p}\|_{C^1(G_{2\rho})} \big[\| \hat{V}_\p [ q_{c,\p} \hat{v}_\p ] \|_{H^1}+\| f \|_{H^1} \big] \nonumber\\
  & \leq C |s|^6 \| f \|_{H^2} + C |s|^8 \| q_{c,\p} \|_{C^1(G_{2\rho})} q_{c,+}^{1/2} \, \| \hat{v}_\p|_D \|_{L_{q_c}^2(D)}. \label{eq:aux7}
\end{align}
Furthermore, restricting~\eqref{eq:aux3} to $D$ shows that $w := \hat{v}_\p|_D \in L^2(D)$ and $f_D := \left. f \right|_D\in L^2(D)$ satisfy
\[
  w+ \frac{s^2}{c_0^2} \left. \left[ \hat{V}_\p \big[ q_{c,\p} \hat{v}_\p \big] \right] \right|_D = f_D \quad \text{in } L^2(D).
\]
We already showed in~\eqref{eq:aux6} that $\hat{V}(f) = \hat{V}_\p(f)$ on $D$ if $\mathrm{supp}(f) \subset \overline{D}$ and conclude that 
$ [ \hat{V}_\p [ q_{c,\p} \hat{v}_\p ] ] \big|_D =   [\hat{V}(q_c w) ] \big|_D$.
Thus, the bound~\eqref{solBound} from Theorem~\ref{solEst} states that $\| w \|_{L_{q_c}^2(D)} \leq C |s| \| f_D \|_{L_{q_c}^2(D)}$, i.e., 
\begin{eqnarray*}
  \| \hat{v}_\p |_D \|_{L_{q_c}^2(D)} 
  &\leq& C |s| \big\| f_D \big\|_{L_{q_c}^2(D)} 
  \leq C |s| \big\| q_{c}^{1/2} f \big\|_{L^2(D)}\\&
  \leq &C |s| q_{c,+}^{1/2} \big\| f \big\|_{L^2(G_{2\rho})}
  \leq C |s| q_{c,+}^{1/2} \big\| f \big\|_{H^2}. 
\end{eqnarray*}
Combining the latter estimate with~\eqref{eq:aux7} shows that
\begin{equation} \label{eq:aux5}
  \| \hat{v}_\p \|_{H^2} \leq C_B(q_c) \, |s|^9 \| f \|_{H^2}, 
\end{equation}
and that~\eqref{eq:cruxCo} is satisfied if 
$C_A C_B |s|^{13} N^{-1} \| q_{c,\p} \|_{H^r} < 1/2$.
This implies that the number $N_0=N_0(s)$ is bounded from below by $c|s|^{13}$ for some constant $c=c(q_c, \sigma_0)$ and 
\begin{equation} \label{eq:boundInve}
   \bigg\| \bigg[I + \frac{s^2}{c_0} \hat{V}_\p (Q_N (q_{c,\p} \, \cdot)) \bigg]^{-1} \bigg\|_{H^2\to H^2}
   \leq 2 \bigg\| \bigg[I + \frac{s^2}{c_0} \hat{V}_\p( q_{c,\p} \, \cdot ) \bigg]^{-1} \bigg\|_{H^2\to H^2}
   \leq 2 C_B |s|^9.
\end{equation}

To get the claimed error estimate, we introduce the orthogonal projection $P_N: \, H^t \to T_N$ onto $T_N$ and recall first that $\| (I-P_N) v \|_{H^\ell} \leq C N^{-t+\ell} \| v \|_{H^t}$ for $v \in H^t$ and $0 \leq \ell \leq t$, and second that $\| (P_N - Q_N) v \|_{H^\ell} \leq C N^{-t+\ell} \| v \|_{H^t}$ for $v \in H^t$ and $3/2 < \ell \leq t$ (see~\cite{Saran2002}, Sect.~8.3 and 8.5 for proofs in one and two dimensions). Consider the unique solution $\hat{u}_\p \in H^2$ to
\[
  \hat{u}_\p + \frac{s^2}{c_0^2} \hat{V}_\p [q_{c,\p} \hat{v}_\p] 
  = - \frac{s^2}{c_0^2} \hat{V}_\p [q_{c,\p} \hat{u}^i_\p].
\]
Exploiting that $\hat{V}_\p$ diagonalizes on trigonometric polynomials we deduce that $\hat{V}_\p$ commutes with the orthogonal projection $P_N$ and obtain that 
\begin{align*}
  \left[ I + \frac{s^2}{c_0^2}\hat{V}_\p Q_N (q_{c,\p} \, \cdot) \right] [\hat{u}_{\p,N} - \hat{u}_\p]
  & = - \frac{s^2}{c_0^2} \hat{V}_\p  Q_N (q_{c,\p} \hat{u}^i_\p)  
  - \hat{u}_\p
  + \frac{s^2}{c_0^2} \hat{V}_\p Q_N (q_{c,\p} \hat{u}_\p) \\
  & = - \frac{s^2}{c_0^2} P_N \hat{V}_\p (q_{c,\p} \hat{u}^i_\p)  
  + \frac{s^2}{c_0^2} P_N \hat{V}_\p (q_{c,\p} \hat{u}_\p)
  - \hat{u}_\p \\
  & \quad - \frac{s^2}{c_0^2} \hat{V}_\p \left[ (Q_N-P_N) (q_{c,\p} (\hat{u}^i_\p - \hat{u}_\p)) \right] \\
  & =  P_N \hat{u}_\p - \hat{u}_\p 
  - \frac{s^2}{c_0^2} \hat{V}_\p \left[ (Q_N-P_N) (q_{c,\p} (\hat{u}^i_\p-\hat{u}_\p)) \right].
\end{align*}
Inverting the operator on the left in $H^2$ and exploiting that
$\| q_{c,\p} \hat{u}_\p\|_{H^2} = \| q_{c} \hat{u} \|_{H^2(D)}$ where $\hat{u}$ 
solves~\eqref{eq:hatUEq} yields the error estimate 
\begin{align}
  \| \hat{u}_{\p,N} - \hat{u}_\p \|_{H^2}
  & \stackrel{\eqref{eq:boundInve}}{\leq} 2 C_B |s|^9
  \left( \| P_N \hat{u}_\p - \hat{u}_\p \|_{H^2} + \frac{|s|^2}{c_0^2} \| \hat{V}_\p (Q_N-P_N) (q_{c,\p} (\hat{u}^i_\p- \hat{u}_\p)) \|_{H^2} \right) \nonumber \\
  & \leq C |s|^9 \left( N^{-t} \| \hat{u}_\p \|_{H^{2+t}} + |s|^4 \| (Q_N-P_N) (q_{c,\p} (\hat{u}^i_\p- \hat{u}_\p)) \|_{H^1}  \right) \nonumber \\
  & \leq C |s|^9 \left( N^{-t} \| \hat{u}_\p \|_{H^{2+t}} + |s|^4 N^{-r} \| q_{c,\p} \hat{u}^i_\p \|_{H^r} + |s|^4 N^{-1} \| q_{c,\p} \hat{u}_\p\|_{H^2} \right)  \label{eq:aux9} \\
  & \leq C |s|^9 \left( N^{-t} \| \hat{u}_\p \|_{H^{2+t}} + |s|^4 N^{-r} \| q_{c,\p} \hat{u}^i_\p \|_{H^r} + |s|^4 N^{-1} \| q_{c} \hat{u} \|_{H^2(D)} \right)  \nonumber \\
  & \stackrel{\eqref{eq:aux11}}{\leq} C |s|^9 \left( N^{-t} \| \hat{u}_\p \|_{H^{2+t}} + |s|^4 N^{-r} \| q_{c,\p} \hat{u}^i_\p \|_{H^r}\right.\nonumber\\&\qquad\left. + |s|^9 N^{-1} \| q_{c} \|_{C^2(D)} \| \hat{u}^i \|_{L^2_{q_c}(D)} \right).  \nonumber 
\end{align}
If $q_{c,\p} \hat{u}^i_\p \in H^2$, then $\hat{u}_\p = - (s^2/c_0^2) \hat{V}_\p (q_{c,\p} (\hat{u}^i_\p+ \hat{u}_\p))$ belongs to $H^3$: 
The term $q_{c,\p} \hat{u}^i_\p+ q_{c,\p} \hat{u}_\p$ belongs to $H^2$ since all three  functions belong to this Banach algebra and $\hat{V}_\p$ is bounded from $H^2$ into $H^3$. Moreover, the techniques used in~\eqref{eq:aux9} show that
\begin{align}
  \| \hat{u}_\p \|_{H^3} 
  & \leq \frac{|s|^2}{c_0^2} \| \hat{V}_\p (q_{c,\p} (\hat{u}^i_\p+ \hat{u}_\p)) \|_{H^3} 
  \leq C |s|^4  \big( \| q_{c,\p} \hat{u}^i_\p \|_{H^2} + \| q_{c} \hat{u}_\p \|_{H^2(D)} \big) \nonumber \\
  & \leq C |s|^4  \big( \| q_{c,\p} \hat{u}^i_\p \|_{H^2} + |s|^5 \| q_{c,\p} \|_{C^2(D)}  \| \hat{u}^i \|_{L^2_{q_c}(D)} \big) \nonumber \\
  & \leq C |s|^9  \big( \| q_{c,\p} \hat{u}^i_\p \|_{H^2} + \| q_{c,\p} \|_{C^2(G_{2\rho})}  \| \hat{u}^i \|_{L^2_{q_c}(D)} \big). \nonumber
\end{align}
For higher-order regularity, one considers the equation~\eqref{eq:aux3} in $H^{2+k}$ for $k\in\N$ with a right-hand side $f = (s^2/c_0^2) \hat{V}_\p(q_{c,\p} \hat{u}^i_\p) \in H^{2+k}$ for $q_{c,\p} \hat{u}^i_\p \in H^{1+k}$ and $q_{c,\p} \in C^{2}(G_{2\rho}) \cap H^{1+k}$. Iterating the estimate~\eqref{eq:aux13} shows that 
\[
  \| \hat{u}_\p \|_{H^{2+k}} 
  \leq 
  C(k) [|s|^4 \| q_{c,\p} \hat{u}^i_\p \|_{H^{1+k}} + |s|^{4(k-1)} \| q_{c,\p} \|_{H^{1+k}}^{k-1} \| q_{c,\p} \hat{u}_\p \|_{H^2}], \quad k=1,2,\dots.
\] 
In~\eqref{eq:aux9} we already exploited that $\| q_{c,\p} \hat{u}_\p \|_{H^2} \leq |s|^5 \| q_c \|_{C^2(G_{2\rho})} \| \hat{u}^i \|_{L^2_{q_c}(D)}$. Thus, 
\[
  \| \hat{u}_\p \|_{H^{2+k}} \leq C \left[ |s|^4 \| q_{c,\p} \hat{u}^i_\p \|_{H^{1+k}} 
  + |s|^{4(k-1)+5} \| q_{c,\p} \|_{H^{1+k}}^{k-1} \| q_c \|_{C^2(G_{2\rho})} \| \hat{u}^i \|_{L^2_{q_c}(D)} \right], \quad k\in\N. 
\]
In combination with~\eqref{eq:aux13}, these regularity estimates show that 
\begin{align*}
  & \| \hat{u}_{\p,N} - \hat{u}_\p \|_{H^2}
  \leq C |s|^9 \left( N^{-k} \| \hat{u}_\p \|_{H^{2+k}} + |s|^4 \| (Q_N-P_N) (q_{c,\p} (\hat{u}^i_\p- \hat{u}_\p)) \|_{H^1}  \right) \\
  & \leq C |s|^9 N^{-k} \left( \| \hat{u}_\p \|_{H^{2+k}} + |s|^4 \| q_{c,\p} \hat{u}^i_\p \|_{H^{1+k}} 
  + |s|^4 \| q_{c,\p} \|_{H^{1+k}} \| \hat{u}_\p \|_{H^{1+k}} \right) \\
  & \leq C |s|^{13} N^{-k} (1+\| q_{c,\p} \|_{H^{1+k}}) \left( \| \hat{u}_\p \|_{H^{2+k}} + \| q_{c,\p} \hat{u}^i_\p \|_{H^{1+k}} \right) \\
  & \leq C |s|^{18+4(k-1)} N^{-k} (1+\| q_{c,\p} \|_{H^{1+k}}) \left( \| q_{c,\p} \hat{u}^i_\p \|_{H^{1+k}} 
  + \| q_{c,\p} \|_{H^{1+k}}^{k-1} \| q_c \|_{C^2(G_{2\rho})} \| \hat{u}^i \|_{L^2_{q_c}(D)}  \right).
\end{align*}
\end{proof}

We now use Lemma 5.5 of \cite{lubich} to prove an error estimate for Collocation Convolution Quadrature (CCQ).
Assume that $\hat{u}^i_\p \in H^2$, that $q_{c,\p} \in H^2 \cap C^2(G_{2\rho})$, and define the operator 
$\cKN:H^2\to H^2$ by
\[
\cKN(\hat{u}^i_\p):=\hat{u}_{\p,N} - \hat{u}_\p
\]
where $\hat{u}_{\p,N}$ and $\hat{u}_\p$ are defined in Lemma \ref{difficult}. According to this lemma there is a constant $C_1$ independent of $s$ and $N$ such that if $N^{-1}|s|^{13} < C_1$ this operator is well defined and 
\[
\Vert \cKN\Vert_{H^2\to H^2}\leq C |s|^{18} N^{-1} 
\]
To apply Lemma 5.5 of \cite{lubich} letting $\epsilon= (C_1N)^{-1/13}$ we have that $\cKN$ is bounded and analytic on the half disks 
\[
\{\Re(s)>\sigma_0>0\}\cap \{|\epsilon s|<1\}
\]
So if $B=\sup_{|\zeta|<1}\delta(\zeta)$ and $\Delta t >  B (C_1 N)^{-1/13}$ then by Lemma 5.5 of \cite{lubich} we have the estimate
\begin{align}
\left(\Delta t\sum_{m=0}^M\Vert ({\cal K}_{\p,N}(\delta_t^{\Delta t})u^i)(m\Delta t)\Vert_{L^2_{q_c}(D)}^2\right)^{ {1/2} } 
& \leq C \left(\Delta t\sum_{m=0}^M\Vert ({\cal K}_{\p,N}(\delta_t^{\Delta t})u^i)(m\Delta t)\Vert_{H^2}^2\right)^{ {{1/2}} }  \nonumber\\
& \leq C N^{-1}\Vert u^i\Vert_{H^{18}_{\sigma_0}(\real_+,H^2)} \label{pertest}
\end{align}
for $\Delta t$ small enough.

We can now prove the main theorem of the paper. Let $u^{s,\Delta t,N}$ denote the fully discrete
convolution quadrature solution corresponding to applying convolution quadrature to the integral equation (\ref{IEColl})  for $\hat{u}^s_{\p,N}$ and having time step values $u^{s,\Delta t,N}_m\in T_N$, $m=0,1,\dots$\,.

\begin{theorem} \label{mainT} Let the conditions on $\delta(\xi)$ of Theorem \ref{semid} hold and suppose $N$ and $\Delta t$ are chosen so that $\Delta t >  B (C_1 N)^{-1/13}$. Suppose $q_{c,\p}\in H^2 \cap C^2(G_{2\rho})$. Then, provided $u^i\in H^{18}_{\sigma_0}(\real_+,H^2(D))$ for all $N$ large enough and $\Delta t$ small enough where $\Delta t=T/M$,
\[
\sum_{m=0}^M\left(\Delta t\Vert  u^s(t_m,\cdot)-u^{s,\Delta t,N}_m\Vert_{L_{q_c}^2(D)}^2\right)^{1/2}  
\leq
C(\Delta t^p+N^{-1})
\]
where $C$ depends on $T$ but not on $\Delta t$ and $N$ where $p$ is the order of the A-stable time stepping method.
\end{theorem}
\begin{remark} The requirement that $N(\Delta t)^{13}$ be large enough is a crushing 
stability condition.  We do not see this in practice as we shall show in the next section.  
Improvement of this bound would be highly desirable.
\end{remark}

\begin{proof}
Let $u^{s,\Delta t}$ denote the semi-discrete
convolution quadrature solution corresponding to applying convolution quadrature to the integral equation for $\hat{u}^s$ (see (\ref{ushateq})) and having time step values $u^{s,\Delta t}_m\in L^2_{q_c}(D)$,  $m=0,1,\dots$\,. Of course
\[
\Vert u^s(t_m,\cdot)-u^{s,\Delta t,N}_m\Vert_{L_{q_c}^2(D)}\leq  \Vert u(t_m,\cdot)-u^{s,\Delta t}_m\Vert_{L_{q_c}^2(D)}+C\Vert u^{s,\Delta t}_m-u^{s,\Delta t,N}_m\Vert_{L^2(D)}.
\]
But under the conditions of Theorem \ref{semid} the first term on the right hand side is estimated by $(\Delta t)^p$ since the periodized solution and the original solution agree on $D$ before discretization.  The second term above is estimated by (\ref{pertest}).
\end{proof}

\section{Numerical Results}\label{numerics}
In this section we provide a preliminary numerical experiment designed to test observed convergence rates compared to those predicted in Theorem \ref{mainT}, and also test if the stability constraint in that theorem is sharp.  Of course with only one special test problem, our results are at most indications for further work.  The results are computed  in two spatial dimensions.

The problem we shall solve is to compute the scattering of the incident field
\[
u^i(x,t)=\sin(a(t-x_1/c_0))\exp(-b(t-x_1/c_0)^2),\quad a=4,\, b=1.4,\, c=2,
\]
from a circular domain  of radius 0.275.  Note that the incident wave is not perfectly zero at $t=0$ in $D$. 
The external domain has sound speed $c_0=1$ and $c(x)=\sqrt{2}$ in $D$ so this also tests if the method works for $q_c<0$.  The final time equals $T=4$.  

Using the method of \cite{Banjai_2008} we choose $M$ and $N$ and solve (\ref{IEColl}) for
\[
s=s_m=\frac{\delta(\lambda \xi^m)}{\Delta t},\quad m=0,\dots, M,
\]
where $\delta(\xi)=(\xi^2-4\xi+3)/2$ corresponds to the choice of BDF2 in (\ref{multi}),
\[
\xi= \exp(-2\pi/(M+1))
\]
and where $\lambda<1$ is a stability parameter chosen as described in Remark 5.11 of \cite{Banjai_2008}. It is possible
that the choice of this parameter further stabilizes the scheme making the stability constraint in Theorem \ref{mainT}
unnecessary.  
The incident field $\hat{u}^i$ is now chosen to be
\[
\hat{u}^i_{\lambda,m}=\sum_{j=0}^M\lambda^ju^i(x,t_j)\xi^{jm}, \quad m=0,\dots,M.
\]
This gives us $M+1$ approximate solutions $\hat{u}^s_{\p,N,m}\in T_N$, $m=1,\dots, M+1$ where $T_N$ is the space of two dimensional trigonometric polynomials on $G_{2\rho}$ with $\rho=0.275$ of degree $N$ in each direction.

An important point is that because (\ref{IEColl})  is of the second kind, we can compute the solutions $\hat{u}^s_m$ using 
a two grid iterative method (see \cite{hohage01}).  

Once the Fourier-Laplace modes are known the time steps 
can be obtained approximately via
\[
u^{s,\Delta t,N,\lambda}_m=\frac{\lambda^{-m}}{M+1}\sum_{j=0}^M\hat{u}^s_{\p,N,m}\xi^{-mj}.
\]
Here the superscript emphasizes that this quantity is an approximation to the field $u^{s,\Delta t,N}$ analyzed earlier.
For a detailed discussion of this approach see \cite{Banjai_2008}.  This includes a detailed derivation of the approximate equivalence of this solution with the original convolution quadrature solution computed by marching on in time for a boundary integral equation.

Note that  $q_{c,\p}\in H^0$ so we would expect to see slower than the $O(N^{-1})$  spatial convergence  rate from Theorem \ref{mainT}.

In our first experiment we choose a fixed value of the spatial parameter $N$ and increase the number of time-steps $M$.  The results are shown in Fig.~\ref{fig1}.  For coarse time discretizations we expect the error to be dominated by the time stepping error. As expected, convergence is second order in $\Delta t$ until a minimum error is reached, presumably due to the fixed spatial mesh. The minimum error decreases with $N$ . After this the error rises gradually, perhaps reflecting
the need for a stability restriction between $N$ and $\Delta t$.  This should be investigated more thoroughly in
three dimensions, but this outside the scope of this paper.

\begin{figure}
\begin{center}
\resizebox{0.7\textwidth}{!}{\includegraphics{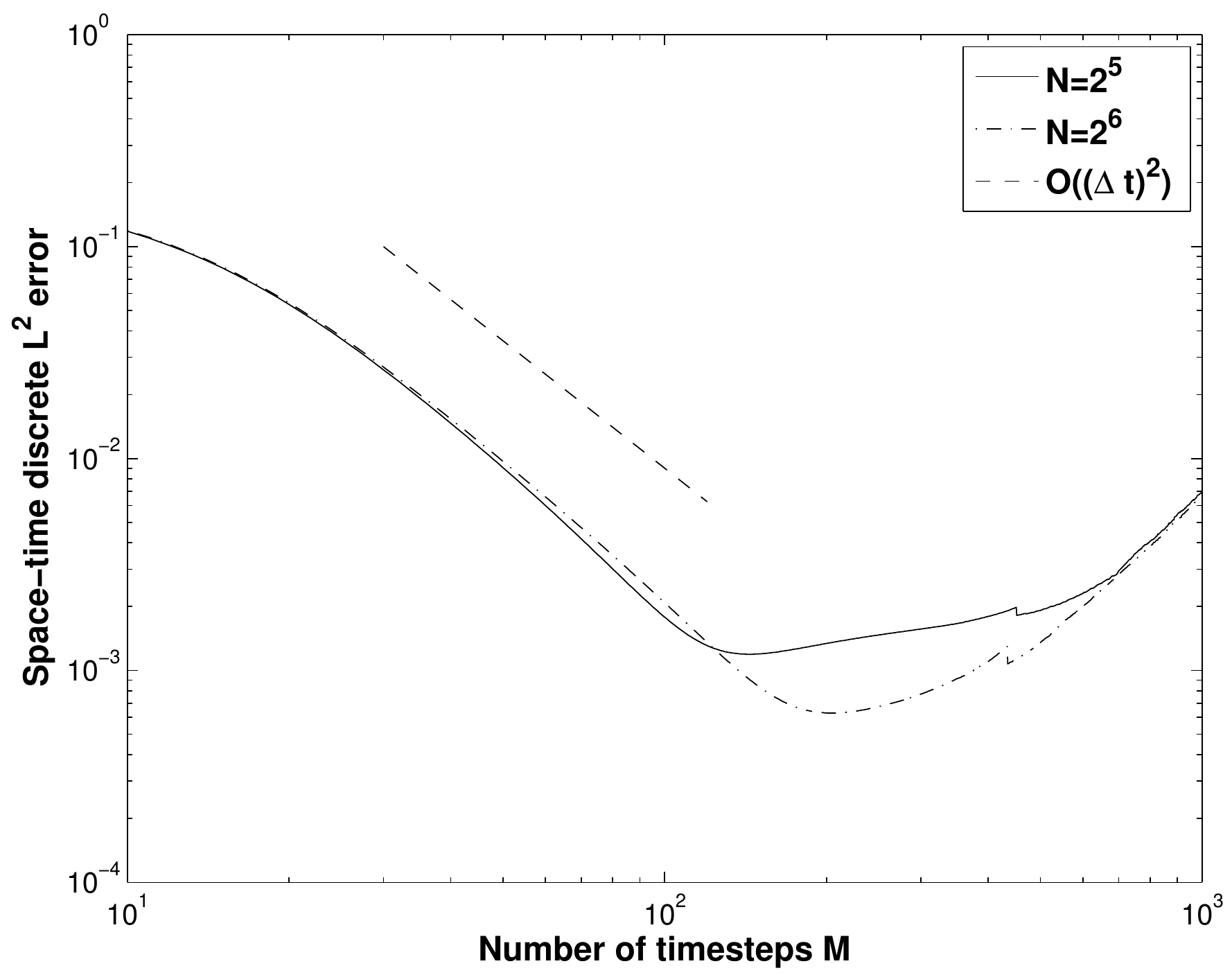}}
\end{center}
\caption{Error as a function of the number of time steps $M$ for two fixed spatial meshes $N=2^5$ and $N=2^6$.}
\label{fig1}
\end{figure}

Our second result in Fig.~\ref{fig2}  shows convergence for two fixed numbers of time steps $M$ as the spatial discretization parameter varies from $M=2^2,\dots, 2^9$.   We see rapid initial convergence reminiscent of pre-asymptotic convergence
for finite difference methods, then a period of roughly $O(N^{-1})$ convergence and finally a plateau presumably
due to the error from time discretization.  Unlike Fig.~\ref{fig1} the error does not rise markedly as $N$ increases after reaching a minimum value.  This is consistent with our theory in that the stability constraint involves a lower bound on $\Delta t$ but $N$ is free to increase without bound.

\begin{figure}
\begin{center}
\resizebox{0.7\textwidth}{!}{\includegraphics{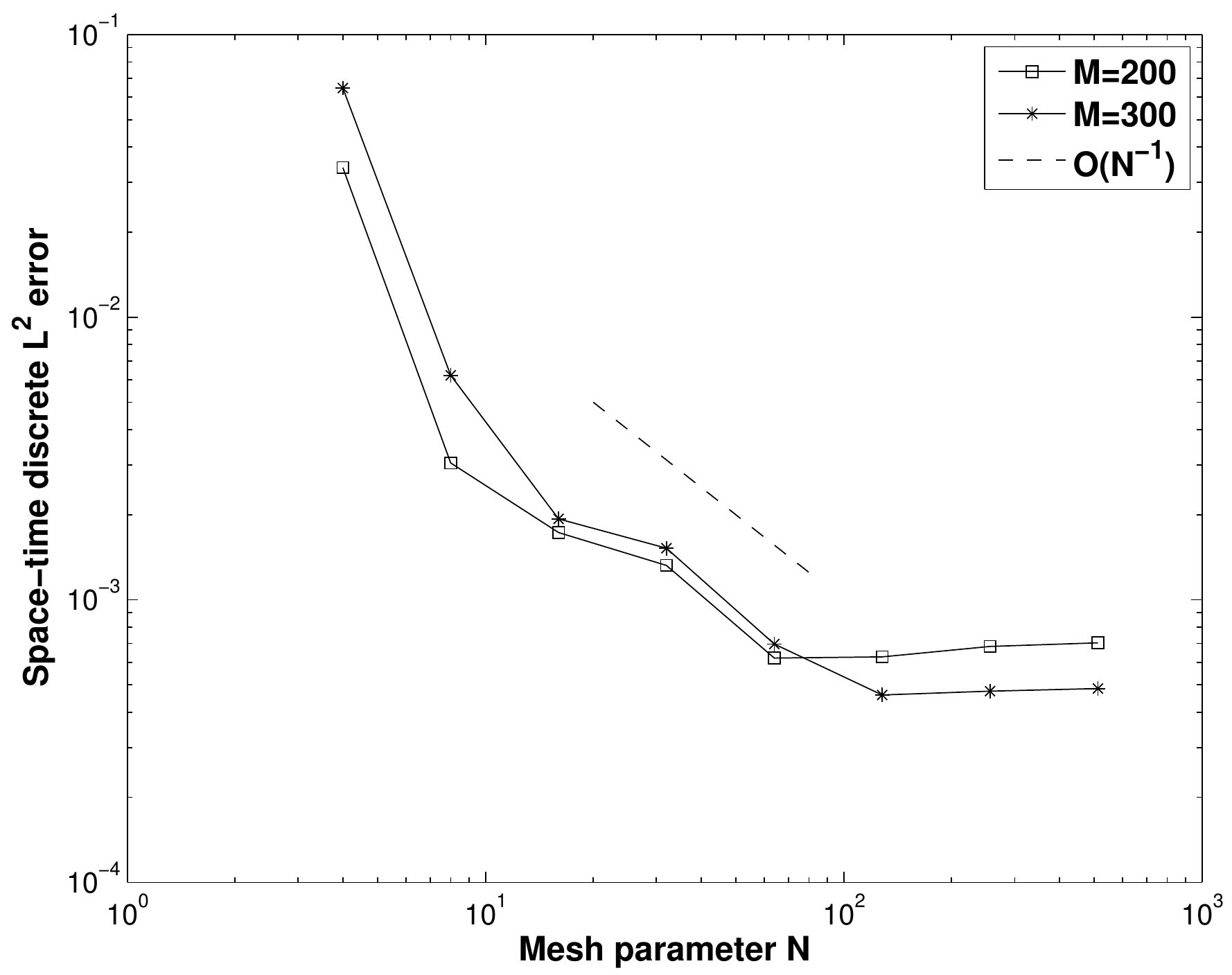}}
\end{center}
\caption{Error as a function of the spatial grid parameter $N$ for two fixed time steps with $M=200$ and $M=300$.}
\label{fig2}
\end{figure}

\section{Conclusion}\label{concl}
We have presented some basic theory for a time domain volume integral equation appropriate for the wave equation.
This equation is coercive in an appropriate norm, and hence Convolution Quadrature can be applied to the Galerkin equations.  Instead we apply a collocation scheme to discretize in space, and using a perturbation argument to
obtain a convergence result with a very strong stability constraint.  Our numerical results suggest that this constraint its not active (at least for our simple example).  Clearly a much more thorough program of numerical testing is needed, and it would also be desirable to test a Galerkin scheme based on periodized trigonometric polynomials to avoid any question of stability constraints, and possibly  improve the regularity requirements.
\section{Acknowledgements}
The research of A.L. is supported in part by an exploratory project granted by the University of Bremen in the framework of its institutional strategy, funded by the excellence initiative of the federal and state governments of Germany.

The research of P.M.  is supported in part  by US NSF grant number DMS 1114889 and AFOSR grant number FA9550-13-1-0199.

\end{document}